\documentclass[oneside,11pt]{amsart}
\usepackage{mathptmx}
\usepackage{amsthm}
\usepackage[scale={.7,.8}]{geometry}
\usepackage{mathrsfs}
\usepackage{mathptmx,amssymb}
\usepackage{amsmath}
\usepackage{color}
\definecolor{BLUE}{RGB}{41,86,143}
\definecolor{RED}{RGB}{178,31,53}
\usepackage{graphicx}

\title{MULTIDIMENSIONAL RANDOM WALKS CONDITIONED TO STAY ORDERED VIA GENERALIZED LADDER HEIGHT FUNCTIONS}
\usepackage[dvipsnames]{xcolor}
\usepackage[colorlinks,citecolor=Plum,urlcolor=Periwinkle,linkcolor=DarkOrchid,backref=page]{hyperref}
\usepackage{color}
\usepackage{subcaption}
\newtheorem{teo}{Theorem}
\newtheorem{lemma}{Lemma}

\newtheorem{propo}{Proposition}

\newtheorem{remark}{Remark}

\newcommand{\bb}[1]{\mathbb{#1}}
\newcommand{\E}{\ensuremath{ \bb{E} } }

\newenvironment{lesn}{\begin{linenomath}\begin{equation*}}{\end{equation*}\end{linenomath}}

\usepackage[colorlinks,citecolor=Plum,urlcolor=Periwinkle,linkcolor=DarkOrchid]{hyperref}
\usepackage[mathlines]{lineno}

\newcommand{\re}{\ensuremath{\mathbb{R}}}
\newcommand{\weyl}{\ensuremath{\mathbb{W}}}
\newcommand{\paren}[1]{\ensuremath{\left( #1\right) }}
\newcommand{\F}{\ensuremath{\mathscr{F}}}

\newcommand{\p}{\mathbb{P}}
\newenvironment{esn}{\begin{equation*}}{\end{equation*}}

\newcommand{\se}{\ensuremath{\mathbb{E}}}

\newcommand{\z}{\ensuremath{\mathbb{Z}}}
\newcommand{\na}{\ensuremath{\mathbb{N}}}

\newcommand{\indi}[1]{\si\left\{ #1\right\}}
\newcommand{\si}{{\ensuremath{\bf{1}}}}

\newcommand{\esp}[1]{\ensuremath{\se\! \left( #1 \right)}}

\newcommand{\proba}[1]{\ensuremath{\sip\! \left( #1 \right)}}
\newcommand{\sip}{\mathbb{P}}

\newcommand{\probac}[2]{\ensuremath{\sip\! \left( #1 \, | #2 \right)}}
\newcommand{\bo}[1]{\ensuremath{{\bf #1 } }}
\usepackage{lineno}
\usepackage{algorithm}
\makeatletter
\def\BState{\State\hskip-\ALG@thistlm}
\makeatother
\usepackage{algcompatible}
\author[Angtuncio-Hern\'andez]{Osvaldo Angtuncio-Hern\'andez}
\address[OAH]{Instituto de Matem\'aticas, Universidad Nacional Aut\'onoma de M\'exico.  
	Distrito Federal CP 04510,  M\'exico}
\thanks{Research supported by CoNaCyT grant FC-2016-1946 and UNAM-DGAPA-PAPIIT grant IN115217.} %%% Revisar clave
\email{osvaldo.angtuncio@matem.unam.mx}

\begin{document}
%\tableofcontents

%\begin{titlepage}
	
\maketitle

%%%%%%%%%%%%%%%%%%%%%%%%%%%%%%%%%%%%%%%%%%%%%%%%%%%%%
%\newpage

%\center{MULTIDIMENSIONAL RANDOM WALKS CONDITIONED TO STAY ORDERED VIA GENERALIZED LADDER HEIGHT FUNCTIONS}

\begin{abstract}
%In this paper we define a $d$-dimensional random walk conditioned to have ordered components forever, for any $d\in \na$, under minimal assumptions. 
%This is done using (sub)harmonic functions and the associated Doob $h$-transforms. 
%The main tool is to construct a \emph{ladder height function} for the random walk, which is based on a generalization of the ladder times in the unidimensional case.
%The construction is equivalent as the limit as $q\downarrow 0$ of the random walk conditioned to have ordered 	components up to a geometric time of parameter $1/q$. 
%Our construction works with minimal hypotheses, in particular, the components of $X$ could take values on $\re^d$, be dependent, have different distributions, or have infinite mean. 

Random walks conditioned to stay positive are a prominent topic in fluctuation theory. 
One way to construct them is as a random walk conditioned to stay positive up to time $n$, and let $n$ tend to infinity. 
A second method is conditioning instead to stay positive up to an independent geometric time, and send its parameter to zero.
%; this approach uses excursion theory. 
The multidimensional case (condition the components of a $d$-dimensional random walk to be ordered) was solved in \cite{MR2430709} using the first approach, but some moment conditions need to be imposed.
Our approach is based on the second method, which has the advantage to require a minimal restriction, needed only for the finiteness of the $h$-transform in certain cases.
We also characterize when the limit is Markovian or sub-Markovian, and give several reexpresions of the $h$-function.
Under some conditions given in \cite{2018arXiv180305682I}, it can be proved that our $h$-function is the only harmonic function which is zero outside the Weyl chamber $\{x=(x_1,\ldots, x_d)\in \re^d: x_1<\cdots < x_d\}$.
\end{abstract}
%\end{titlepage}

\section{Introduction and main results}

%Let $X^1,\ldots, X^d$ be independent, simple, symmetric random walks on $\mathbb{Z}$.
%Let $\p_{\bo{i}}$ be the probability measure of $X=(X^1,\ldots,X^d )$ starting at $\bo{i}=(i_1,\ldots, i_d)$. 

\subsection{Motivation}

Let $X^1,\ldots, X^d$ be independent, simple, symmetric random walks on $\mathbb{Z}$.
Let $\p_{\bo{i}}$ be the probability measure of $X=(X^1,\ldots,X^d )$ starting at $\bo{i}=(i_1,\ldots, i_d)$. 
In \cite{MR1887625} and \cite{MR2430709}, the authors define $X$ \emph{conditioned to have ordered components}.
The interest on such processes is by their relation with random matrix theory, for example, with Dyson's Brownian motion \cite{MR0148397}, which can be interpreted as $d$ Brownian motions conditioned to stay ordered at all times.
As another important connection, conditioning a 2-dimensional random walk to have ordered components, is equivalent to condition a random walk to stay non-negative, a theory with a long history (see \cite{MR1001739,MR1159575,MR1232850,MR1787140,MR1331218,
	%MR1329107,B.M.Hambly2001,MR1942425,MR2197112,
	MR1844435,MR2073341,MR2164035,MR2451576,MR2447836,
	%MR3068391,
	grama2016limit}).
%MR2670195,
The conditioning event can be written as
\begin{lesn}
	A=\{X^1_j<\cdots<X^d_j\mbox{ for all }j\geq 0 \}.
\end{lesn}Denoting by $\weyl=\{x\in \re^d:x_1< \cdots < x_d \}$ the \emph{Weyl chamber},
the conditioning event $A$ can be rewritten as $\{X_j\in \weyl,\forall j\geq 0 \}$.
Note that, even in the case $i_1<\cdots <i_d$, we have
\begin{lesn}
	\p_{\bo{i}}(A)=0,
\end{lesn}since $X^2-X^1$ is an oscillating random walk. 
Therefore, a rigorous definition of the law of $X$ conditioned on $A$ should be given. 
This is done in \cite{MR1887625} and \cite{MR2430709} (as a particular case of their results), introducing the event
\begin{lesn}
	A_n=\{X^1_j< \cdots < X^d_j\mbox{ for all }j\in [n] \}\ \ \ \ \ \ \ \ \ n\in \na,
\end{lesn}with $[n]=\{1,\ldots, n \}$, and proving that for every $k\in \na$, the limit as $n\to \infty$
\begin{lesn}
	\p_{\bo{i}}\paren{X_0=\bo{i}_0,\ldots, X_k=\bo{i}_k|A_n}
\end{lesn}exists and is a probability measure.
In fact, Karlin-McGregor's formula (cf. \cite{MR0114248}) gives us an expression for $\p_{\bo{i}}(A_n)$ and implies
\begin{lesn}
	\lim_n\p_{\bo{i}}\paren{X_0=\bo{i}_0,\ldots, X_k=\bo{i}_k|A_n}=\E_{\bo{i}}\paren{\frac{\Delta(\bo{i}_k)}{\Delta(\bo{i}_0)}\indi{X_0=\bo{i}_0,\ldots, X_k=\bo{i}_k}}, 
\end{lesn}where, for $x=(x_1,\ldots, x_d)\in \mathbb{W}$
\begin{equation}\label{eqnVandermondeDeterminant}
\Delta(x)=\prod_{1\leq i<j\leq d}(x_j-x_i)=\det \paren{\paren{x^{i-1}_j},i,j\in [d]}.
\end{equation}is the Vandermonde's determinant. 
Hence, the conditioning is made using Doob's \emph{$h$-transform}.
Similar transformations, also called \emph{$h$-transforms} or \emph{$h$-process}, appear in \cite{MR1463727}.
We must emphasize that most of the papers constructing ordered random walks are based on finding the limit as $n$ goes to infinity of $\p_{\bo{i}}(A_n)/\p(A_n)$.
Such limit will be the associated \emph{$h$-function} of the process.

The objective is to generalize known constructions to $d$-dimensional random walks conditioned to have ordered components. 
In particular, the components of $X$ could be dependent or have different distributions. 
%\comentario{The most general construction is given in \cite{MR2430709}, under the hypotheses of i.i.d. components, some moment assumption and the central limit theorem.
%In \cite{MR2609589}, the authors found the optimal moment, and also remove the central limit theorem assumption.}
Some models in the literature are \cite{MR2430709,MR3163211,MR3342657,MR3512425,2018arXiv180305682I}.
A general construction (when the drift is zero) is given in \cite{MR3342657}, for random walks in cones. 
The assumptions on the step distribution are that each of its components has mean zero, variance one, and zero covariance between components; also, a moment assumption is made on the step distribution.
The case of non-zero drift was solved in \cite{MR3163211}, using a Cram\'er condition.
Also, the recent paper \cite{2018arXiv180305682I} constructs ordered Markov chains without moment conditions, but the state space must be countable.

Our result has minimal assumptions. 
Define $Y=(X^2-X^1,X^3-X^2,\ldots, X^d-X^{d-1})$. 
In order to avoid trivial cases, we assume $Y$ has components taking negative and positive values with positive probability; 
besides that, the construction works with no further hypotheses when some component $Y^k$ drifts to $-\infty$, or every component $Y^k$ drifts to $+\infty$ (see Lemma \ref{lemmahUpArrowTransformWhenSomeComponentDriftsToMinusInftyOrAllToInfty}). 
In the remaining cases, the assumption on $Y$ is the existence of positive $\epsilon_1,\ldots, \epsilon_{d-1}$ such that
\begin{equation}\label{eqnHypothesisOrderedRandomWalk}
	\proba{X^2_1-X^1_1\geq \epsilon_1,\ldots, X^d_1-X^{d-1}_1\geq \epsilon_{d-1}}>0.
\end{equation}This condition is used only to prove the finiteness of the (sub)harmonic function $h$.

Our method is to analyze a random walk conditioned to have ordered components up to an independent geometric time $N$ of parameter $1-e^{-c}$, and take the limit as $c\to 0$.
The main tool is to construct a \emph{ladder height function} for the random walk, which is based on a generalization of the ladder times in the unidimensional case.
These ideas are adapted from the unidimensional case given for random walks in \cite{MR1232850}, and for L\'evy processes in \cite{MR2164035,MR2320889}.

\subsection{Statement of the results}\label{secStatementOfTheResultsOrderedRandomWalksIntro}

For ease of notation, our results are stated for $d=3$ and for a random walk having state space $\re^d$. 
Let $X=(X^1,X^2,X^3)$ be a 3-dimensional random walk on $\re^3\cup \{\dagger\}$, starting at $X_0=0$, having lifetime $\zeta=\sup\{n:X_n\neq\dagger \}$.
Its increments are denoted by $W=(W^1,W^2,W^3)$, and $W_1$ has law $\p$.
We denote by $Y=(Y^1,Y^2)=(X^2-X^1,X^3-X^2)$ the size of the gap between components, and $y=(x_2-x_1,x_3-x_2)$ for $x\in W$. 
The law of $X$ killed at time $n\in\na$, that is, on the event $\zeta=n$, will be denoted by $\p^n$. 
When killing $X$ at an independent geometric law $N\in \{0,1,\ldots, \}$ with parameter $1-e^{-c}$, its law will be $\p^c=\sum_0^\infty e^{-cn}(1-e^{-c})\p^n $.
The $\sigma$-algebra considered will be $\F_n=\sigma(X_1,\ldots, X_n)$.

The notation that we use is component-wise, hence $\min\{Y_i,i\in \mathcal{I} \}=(\min\{Y^1_i,i\in \mathcal{I} \},\min\{Y^2_i,i\in \mathcal{I} \})$, for any index set $\mathcal{I}\subset \z_+$. 
We define for $[n]=\{1,2,\ldots, n \}$ and $[n]_0=\{0,1,\ldots, n\}$, the processes $\underline{Y}_n=\min\{Y_i,i\in [n]_0 \}$,  $\underline{\underline{Y}}_n=\min\{Y_i,i\in [n] \}$, $\overline{Y}_n=\max\{Y_i,i\in [n]_0 \}$, and $Y_n\vee (y_1,y_2)=(Y^1_n\vee y_1,Y^2_n\vee y_2)$, where $y_1,y_2\in\re$.
We also put $(x_1,\ldots,x_d)<(y_1,\ldots,y_d)$ whenever component-wise the strict inequality is satisfied.

For $\weyl=\{(x_1,\ldots, x_d)\in \re^d:x_1<\cdots < x_d \}$, a \emph{positive regular function} or \emph{harmonic function}, with respect to the transition kernel of $X$ to $\weyl $, is a function $h:\weyl\mapsto \re_+$ such that
\begin{lesn}
	\E_x\paren{h(X_1);\tau>1 }=h(x)\ \ \ \ \ \ \ \ \ x\in \weyl,
\end{lesn}where
\begin{lesn}
	\tau:=\min\{n:X_n\notin \weyl\}.
\end{lesn}A function $h$ is \emph{subharmonic} (\emph{superharmonic}) if $\E_x\paren{h(X_1);\tau>1 }\leq h(x)$ ($\E_x\paren{h(X_1);\tau>1 }\geq h(x)$) for every $x\in \weyl$. 
The resulting (sub)harmonic function associated with a Doob $h$-transform will also be called $h$-function.

To avoid trivial cases, we assume that $Y$ has components taking positive and negative values with positive probability. 
Besides that, the construction works with no further hypothesis if either some component of $Y$ drifts to $-\infty$, or every component of $Y$ drifts to $+\infty$. 
When such conditions are not satisfied, we need Hypothesis \eqref{eqnHypothesisOrderedRandomWalk}, needed only for the finiteness of the (sub)harmonic function $h$.

Our main result is the following, which justifies our construction can be interpreted as a random walk $X$ conditioned to stay ordered forever.

\begin{teo}\label{teoOrderedRWIsAMarkovChainIntro}
	Let $N$ be a geometric time with parameter $1-e^{-c}$, independent of $X$. 
	Assume that 
	\begin{lesn}
		h^{\uparrow}(x):=1+\esp{\sum_{n=1}^{J_1-1}\indi{\overline{Y}_{n-1}-Y_{n}< y}}<\infty\ \ \ \ \  \ \ \ x=(x_1,\ldots, x_d)\in \weyl,
	\end{lesn}with $y=(x_2-x_1,\ldots, x_d-x_{d-1})$ and $J_1=\inf\{n>0: \overline{Y}_{n-1}^k<Y_{n}^k,k\in [d-1]\}$.
	%	$h^\uparrow(x)<\infty$ for every $x\in \weyl$. 
	Then, for every $x\in \weyl$, every finite $\F_n$-stopping time $T$ and $\Lambda\in \F_T$
	\begin{lesn}
		\lim_{c\to 0^+}\p_x\paren{\Lambda,T\leq N|X(i)\in \weyl, i\in [N]}=\p^\uparrow_x(\Lambda, T<\zeta):=\frac{1}{h^\uparrow(x)}\E^Q_x\paren{h^\uparrow(X_T)\indi{\Lambda,T<\zeta}},
	\end{lesn}where $\E^Q_x$ is the expectation under the law of $X$ killed at the first exit time of the Weyl chamber.
	The limit law is a Markov chain with transition probabilities
	\begin{equation}\label{eqnTransitionProbabilitiesOfTheORW}
	p^{\uparrow}(w,dz)=\indi{z\in \weyl}\frac{h^{\uparrow}(z)}{h^{\uparrow}(w)}p(w,dz)\ \ \ \ \ \ \ \ w\in \weyl.
	\end{equation}Moreover, it is a probability measure if $\esp{\tau}=\infty$, or a subprobability measure if $\esp{\tau}<\infty$.	
\end{teo}

We also give simple conditions to ensure $h^\uparrow$ is finite.

\begin{lemma}
	Assume that either
	\begin{enumerate}
		\item some component of $Y$ drifts to $-\infty$, 
		\item every component of $Y$ drifts to $+\infty$ and 
		%$\proba{Y_1>0}>0$
		$\p\paren{\tau=\infty}>0$,
		\item there exists $\epsilon=(\epsilon_1,\ldots, \epsilon_{d-1})\in \re_+^{d-1}$ such that
		\begin{lesn}
			\proba{Y_1>\epsilon}>0.
		\end{lesn}
	\end{enumerate}
	Then
	\begin{lesn}
		h^\uparrow(x)<\infty\ \ \ \ \ \ \ \ \ \ \ \forall\, x\in \weyl.
	\end{lesn}
\end{lemma}

%Its lifetime is $\p^\uparrow_x$-finite if $h^\uparrow$ is subharmonic and $\p^\uparrow_x$-infinite if it is harmonic.

Depending on the drift of its components, we reexpress our function $h^\uparrow$.

\begin{lemma}\label{lemmahUpArrowTransformWhenSomeComponentDriftsToMinusInftyOrAllToInftyIntro}
	Let $x\in \weyl$. 
	If some component of $Y$ drifts to $-\infty$, the $h^\uparrow$-transform is given by
	\begin{lesn}
		h^{\uparrow}(x)=\frac{\E_x\paren{\tau}}{\E\paren{\tau}}.
	\end{lesn}If every component drifts to $+\infty$ and 
	%$\proba{Y_1>0}>0$
	$\p\paren{\tau=\infty}>0$, then
	\begin{lesn}
		h^{\uparrow}(x)=\frac{\p_x\paren{\tau=\infty}}{\p\paren{\tau=\infty}}.
	\end{lesn}
\end{lemma}

We also express $h^\uparrow$ as a renovation function. 
For $k\in [d-1]$, denote by $(\beta^k_i,i\in \na)$ the strict descending ladder times of $Y^k$, that is $\beta^k_0=0$ and for $i\in \na$ the time $\beta^k_i$ is the smallest index $n$ such that $Y^k(\beta^k_{i-1}+n)< Y^k(\beta^k_{i-1})$.
Let $\{\beta_0,\beta_1,\beta_2,\ldots \}$ be the ordered union of all such ladder times, with $\beta_0=1$.
Denoting by $g_n=\beta_n$ and $d_n=\beta_{n+1}$ for $n\geq 0$, the set $\{g_n,g_n+1\ldots, d_n-1 \}$ is the $n$th interval where $\underline{\underline{Y}}$ remains constant. 

\begin{propo}\label{propoReexpresionAsARenewalFunction}
	Let $x\in \weyl$. 
	The $h$-transform can be expressed as 
	\begin{lesn}
		h^{\uparrow}(x)=1+\sum_{n=1}^{\infty}\proba{-\underline{\underline{Y}}_{\beta_n}<y}.
	\end{lesn}
\end{propo}

%%%%%%%%%%%%%%%%%%%%%%%%%%%%%%%%%%%%%%%%%%%%%%%%%%%%%%%%
The main tool to prove Theorem \ref{teoOrderedRWIsAMarkovChainIntro} is an extension of the ladder times in the unidimensional case. 
The paper \cite{MR1232850}, gives the $h$-transform of a unidimensional random walk $X$ conditioned to stay positive forever, under the unique hypothesis that the walk takes positive and negative values with positive probability (see Theorem 2.3 of such paper).
The explicit formula for the $h$-function in such case is
\begin{equation}\label{eqnhTransformUnidimensionalCase}
h(x)=1+\esp{\sum_1^{\alpha^+_1-1}\indi{-X_n<x}},
\end{equation}where $x\geq 0$ and $\alpha^+_1$ is the index of the first visit to $(0,\infty)$. 
Thus, the function $h$ is a particular case of our function $h^\uparrow$.
In fact, Bertoin's $h$-function can be reexpresed in several ways. 
We review some formulas (see \cite{MR1331218}).
The first hitting times, respectively, in $(-\infty,0)$ and in $[n,\infty)$ are denoted by $\tau=\min\{k\geq 1:X_k<0 \}$ and $\sigma_n=\min\{k\geq 1:X_k\geq n \}$.
Let $(H,T)=((H_k,T_k),k\geq 0)$ be the strict ascending ladder point process of the reflected random walk $-X$. 
That is, we have $T_0=0$ and
\begin{lesn}
	H_k=-X_{T_k}\ \ \ \mbox{ and } \ \ \ T_{k+1}=\min\{j>T_k:-X_j>H_k \}.
\end{lesn}The convention is $H_k=\infty$ if $T_k=\infty$. 
The renewal function associated with $H_1$ is
\begin{lesn}
	V(x)=\sum_{k=0}^\infty \p(H_k\leq x),\ \ \ \ \ x\geq 0.
\end{lesn}This is a non-decreasing right-continuous function. 
But the duality lemma gives us
\begin{lesn}
	V(x)=\esp{\sum_{j=0}^{\sigma_0-1}\indi{-x\leq X_j}}=h(x).
\end{lesn}Thus, Proposition \ref{propoReexpresionAsARenewalFunction} is a generalization of this result. 
Our other reexpresions of $h^\uparrow$ given in Lemma \ref{lemmahUpArrowTransformWhenSomeComponentDriftsToMinusInftyOrAllToInftyIntro}, also had their respective reexpresions in the unidimensional case (see \cite{MR1331218}).

%%%%%%%%%%%%%%%%%%%%%%%%%%%%%%%%%%%%%%%%%%%%%%%%%%%%%%%%

We conjecture  that our $h$-transform is subharmonic when $X$ has i.i.d. components taking values in $\re$, satisfies the hypotheses of \cite{MR2430709} or \cite{MR3342657}, and $d>2$.
The reason is that on such papers, the tail of the distribution of $\tau$ is computed, which we prove helps to characterize the harmonicity of our $h$.
In \cite{MR2430709} it is proved that $\p_x\paren{\tau>n}$ is of the order $n^{-d(d-1)/4}$ (see Subsection \ref{sectionORWExpectationOfTau} for another approximations), implying $\E_x(\tau)<\infty$ when $d\geq 3$ and $x\in \weyl$.
%Even that $\E_x(\tau)\leq \E(\tau)$, we expect that $\E(\tau)<\infty$. 
In fact, we prove in Lemma \ref{lemmahUparrowIsSubharmonicOrHarmonic} that $h$ is harmonic (subharmonic) iff the expectation of $\tau$ is infinite (finite). 
This represents a difference with respect to \cite{MR2430709}, since, regardless of the dimension, their $h$-transform is harmonic.
%We give some conditions on the drift of the components of $Y$, to obtain harmonic $h$-transforms.

Nevertheless, such difference has also been observed in \cite{2018arXiv180305682I}. 
On such paper, it is characterized the $h$-transforms of centered irreducible random walks taking values on a countable set, with slowly varying hitting probabilities and other minor assumptions. 
She proved that when $\esp{\tau}=\infty$, any harmonic function for the process is proportional to $V(\cdot)=\E_{\cdot}(\mathcal{T})$, where $\mathcal{T}$ is the exit time of some \emph{ladder height process}, and also that $\lim \p_{\bo{x}}(A_n)/\p(A_n)=V(\bo{x})$, with $\bo{x}\in \mathbb{W}$. 
But, when $\esp{\tau}<\infty$, she proved that $V(\bo{x})\leq \varliminf \p_{\bo{x}}(A_n)/\p(A_n)$ and that $V$ is superharmonic (in our case, our $h$-transform is subharmonic).

%In fact, the set of nonnegative harmonic functions for random walks in $\re^d$, having non-zero drift, and killed when leaving general cones with zero as a vertex, was proved to be uncountable in \cite{MR3283611}. 

It is important to address that even in the unidimensional case, the harmonicity of the $h$-transform highly depends on $\esp{\tau}$ and even in the way we choose the approximating events.
An example appears in \cite{MR1331218} for random walks not drifting to $+\infty$. 
They compare limits of some random walks, under two different conditionings to stay positive. 
It is proved that for oscillating random walks both limits are the same. 
But when the drift is negative, depending on the upper tail of the step distribution it can happen: both limits are the same and the $h$-transform is subharmonic; both limits are different with harmonic $h$-transform.
Results in the same spirit for L\'evy processes, are given in \cite{MR1331218}.
Also, the $h$-transform of the Brownian motion with negative drift conditioned to stay positive is harmonic or subharmonic, depending on the approximation: 
it is proved in \cite{MR1303922} that is harmonic conditioning with $\{\tau>t \}$ and letting $t\to \infty$; while in \cite{MR2164035} is subharmonic when conditioning with $\{\tau>E/c \}$, an exponential random variable with mean 1, and letting $c\to 0$. 

%In the case of ordered random walks, the hypotheses imposed on the step distribution may also lead to a different limit. 
%In \cite{MR2878783}, the authors study a $d$-dimensional random walk $X$ with tails of regular variation of index $-\alpha$, with $\alpha\in (d-2,d-1)$ and $d\geq 4$. 
%Using a linear combination $v$ of the harmonic $h$-transforms for the first and last $d-1$ components of $X$, the authors prove that
%\begin{lesn}
%	U(x)=\sum_0^\infty \esp{v(x+X(n));\tau_x>n}\ \ \ \ x\in \weyl
%\end{lesn}is a strictly positive superharmonic function, that is, $\esp{U(x+X(1));\tau_x>1}<U(x)$ for all $x\in \weyl$. 
%The authors conjecture that there are no harmonic functions for ordered random walks with heavy tails.

This paper is organized as follows. 
%Subsection \ref{subsectionORWKnownResults} is devoted to provide an updated literature and explicit results about ordered random walks.
Our construction is given in Section \ref{subsectionORWConditioningUpToAGeometricIsMarkov}, conditioning the walk to stay ordered up to an independent geometric time. 
We prove this is a Markov chain and an $h$-transform of the process, where the harmonic function is denoted by $h^\uparrow_c$. 
In Subsections \ref{subsectionORWhUparrowCAndLimit} and \ref{subsectionORWPartitioningWithRegenerativeIntervals} we reexpress $h^\uparrow_c$ using a partition of $\na$ on random intervals, making the random walk to be like \emph{excursions} on each interval.
This allows us to obtain in Lemma \ref{lemmaConvergenceOfhUparrowSubc} the limit $h^\uparrow$ of $h^\uparrow_c$ as $c\downarrow 0$, and implies the limit of the random walk is a Markov chain using a change of measure with $h^\uparrow$; this is the second part of Theorem \ref{teoOrderedRWIsAMarkovChainIntro}.
We characterize in Section \ref{sectionPropertiesOfhUparrow} when $h^\uparrow$ is harmonic or subharmonic; give a condition to ensure its finiteness; and prove in Lemma \ref{lemmaRWUnderHTransformIsTheLimitOFConditioningUntilExpTime} that the law of the random walk using the $h$-transform $h^\uparrow$ is the same as the limit of the random walk law conditioned to stay ordered up to a geometric time, which proves Theorem \ref{teoOrderedRWIsAMarkovChainIntro}. 
In Section \ref{sectionORWReexpressionsOfhUparrow} we obtain several reexpresions of $h^\uparrow$. 
Finally, in Section \ref{sectionORWExpectationOfTau} we review known results about the order of $\p_x(\tau>n)$.

\section{The random walk conditioned to be ordered up to a geometric time as an $h$-transform}\label{subsectionORWConditioningUpToAGeometricIsMarkov}

Recall the notation at the beginning of Section \ref{secStatementOfTheResultsOrderedRandomWalksIntro}.
Consider $x=(x_1,x_2,x_3)\in \weyl$ and let $y=(y_1,y_2)=(x_2-x_1,x_3-x_2)$. 
Recall the definition of $\tau=\inf\{n:X_n\notin  \weyl\}$, the first exit time from the Weyl chamber.
For any $n\in \na$ and $A=A_0\times A_1\times \cdots \times A_n\in \mathcal{B}(\re)^{n+1}$, we find the limit as $c\to 0+$ of
\begin{lesn}
	\p^c\paren{\left.\bigcap_0^n\{X_i+x_i\in A_i \}\right|\tau>N}
	=\p^c\paren{\left.\bigcap_0^n\{X_i+x_i\in A_i \}\right|X^1_j+x_1< X^2_j+x_2< X^3_j+x_3,j\in [N]}.
\end{lesn}First we prove this is a Markov chain.

\begin{propo}\label{propoTransitionProbsOfXConditionedUpToAGeo}
	Under $\p^c$ and for any $x=(x_1,x_2,x_3)\in \weyl$, the chain $X+x$ conditioned to be ordered up to time $N$ is a Markov chain with transition probabilities
	\begin{lesn}
		\p^\uparrow_c(w,dy)=\indi{y\in \weyl}\frac{h^\uparrow_c(y)}{h^\uparrow_c(w)}e^{-c}p(w,dy),
	\end{lesn}with $w=(w_1,w_2,w_3)\in \weyl$, $y=(y_1,y_2,y_3)$, $p(w,dy)=\proba{W_1+w\in dy}$ and
	\begin{lesn}
		h^\uparrow_c(w)=\frac{\p^c\paren{X_i+w\in \weyl,i\in [\zeta]}}{\p^c\paren{X_i\in \weyl,i\in [\zeta]}}.
	\end{lesn}
\end{propo}
\begin{proof}
	%The one-step transition probability is
	%\begin{lalign}
	%& p_c^{\uparrow}(w^{\prime},dy^{\prime})\\
	%& = \frac{\p^c\paren{X_1+x^{\prime}\in dy^{\prime},X_0+x^{\prime}=w^{\prime},X_i+x\in \weyl, i\in [\zeta]}}{\p^c\paren{X_0+x^{\prime}=w^{\prime},X_i+x\in \weyl, i\in [\zeta]}}.
	%\end{align*}Taking the sum over all values of $N$, the numerator is equal to
	%\begin{align*}
	%&\sum_{n\geq 1}\p^{n}\paren{X_0=x,X_1+x\in dy,x+X_i-X_0\in \weyl , i\in [n]}\proba{N=n}\\
	% &=\p\paren{X_0+x^{\prime}=w^{\prime},W_1+w^{\prime}\in dy^{\prime}}\\
	%& \times \sum_{n\geq 1}\p\paren{y+X_i-X_1\in \weyl, i\in [n]}\proba{N=n}.
	%\end{align*}Changing the sum to start in 0, and by stationary increments of $X$ under $\p$, the numerator is equal to
	%\begin{lesn}
	%	\p\paren{X_0+x^{\prime}=w^{\prime},W_1+w^{\prime}\in dy^{\prime}}	\times\p^c\paren{y+X_i\in \weyl, 0\leq i\leq N-1,N\geq 1}.
	%\end{lesn}Dividing by $\proba{N\geq 1}$ and by loss of memory, the numerator is
	%\begin{lesn}
	%	\p\paren{X_0+x^{\prime}=w^{\prime},W_1+w^{\prime}\in dy^{\prime}}\indi{y\in \weyl}e^{-c}\times\p^c\paren{y+X_i\in \weyl, i\in [N]}.
	%\end{lesn}Hence, the one-step transition is
	%\begin{align*}
	%& p_c^{\uparrow}(w^{\prime},dy^{\prime})\\
	%& = \p^c\paren{X_1+w^{\prime}\in dy^{\prime}}\indi{y\in \weyl}e^{-c}\frac{y+X_i\in \weyl, i\in [\zeta] }{\p^c\paren{w+X_i\in \weyl, i\in [\zeta]}}.
	%\end{align*}Now, 
	We compute the $n$-step transition probabilities
	\begin{lesn}
		\p^c_x\paren{X_i\in dw_i,i\in [n]_0\ |\ \tau>N},
	\end{lesn} where $w_i\in \re^3$ for $i\in [n]_0:=\{0,\ldots, n\}$ and $w_0=x\in \weyl$. 
	Then, the numerator of the $n$-step transition probability is given by
	\begin{lesn}
		\p^c\paren{X_i+x\in dw_i,i\in[n]_0,X_i+x\in \weyl,i\in [N]}.
	\end{lesn}On such set, for $i\in [n]_0$ we have $X_i+x=w_i\in \weyl$, while for $n+i\in \{n+1,\ldots, N \}$
	\begin{lesn}
		\left\{X_n+x\in dw_n,X_{n+i}+x\in \weyl \right\}= \left\{X_n+x\in dw_n,X_{n+i}-X_n+w_n\in \weyl \right\}.
	\end{lesn}Summing over the values of $N$, and using independent and stationary increments of $X$, the numerator is equal to
	\begin{align*}
	& \indi{\cap_1^n\{w_i\in \weyl\}}\p\paren{X_i+x\in dw_i,i\in [n]_0}\\
	& \times \sum_{m\geq n+1}\p\paren{X_{n+i}-X_n+w_n\in \weyl,i\in [n]}\proba{N=n}\\
	& =\indi{\cap_1^n\{w_i\in \weyl\}}\p\paren{X_i+x\in dw_i,i\in [n]_0}e^{-cn}\\
	& \times \p^c\paren{X_i+w_n\in \weyl,i\in [N]},
	\end{align*}by the lack of memory property of $N$. 
	Therefore, the $n$-step transition probability is given by
	\begin{lesn}
		=  \indi{\cap_1^n\{w_i\in\weyl \}}\p\paren{X_i+x\in dw_i,i\in [n]_0}e^{-cn} \times\frac{\p^c\paren{w_n+X_i\in \weyl ,i\in [N]}}{\p^{c}\paren{x+X_i\in \weyl,i\in [N]}}.
	\end{lesn}
	Denote by $X^\uparrow$ the random walk $X$ conditioned to stay ordered up to time $N$. Considering $w_i\in \weyl$ for $i\in [n-1]_0$, we obtain, using that $X$ is a random walk
	\begin{align*}
	&\p^c\paren{X^\uparrow(n)\in dw_n|X^\uparrow(0)= w_0,X^\uparrow(1)\in dw_1,\ldots ,X^\uparrow(n-1)\in dw_{n-1}}\\
	& =  \indi{\{w_n\in \weyl \}}\p\paren{X_n+x\in dw_n|X_{n-1}+x\in dw_{n-1}}e^{-c}\frac{h_c^{\uparrow}\paren{w_n}}{h_c^{\uparrow}\paren{w_{n-1}}}\\
	& =  \indi{\{w_n\in \weyl \}}\p\paren{X_1\in dw_n|X_0\in dw_{n-1}}e^{-c}\frac{h_c^{\uparrow}\paren{w_n}}{h_c^{\uparrow}\paren{w_{n-1}}},
	\end{align*}which is the one-step transition probability, and depends only on $w_{n-1}$ and $w_n$.
\end{proof}

Now we analyze the function $h^\uparrow_c$.
\subsection[Reexpression of the $h$-function of the ordered RW up to a geometric time]{The $h^\uparrow_c$-function}\label{subsectionORWhUparrowCAndLimit}

A priori, $h^\uparrow_c$ is the division of two probabilities converging to zero. 
We reexpress $h^\uparrow_c$ to prove it converges.
Working with the numerator of $h^\uparrow_c(x)$, first sum over all possible values of $N$
\begin{align*}
& \p^c\paren{X_i+x\in \weyl,i\in [\zeta]}\\
& = (1-e^{-c})\paren{1+\sum_1^\infty e^{-cn}\p^n\paren{X_i+x\in \weyl,i\in [n]}}.
\end{align*}Recall that $Y=(X^2-X^1,X^3-X^2)$ and $y=(x_2-x_1,x_3-x_2)$.
It follows that
\begin{equation}\label{eqnhcUsingTheMinimum}
%\begin{split}
\p^c\paren{X_i+x\in \weyl,i\in [\zeta]}/(1-e^{-c})-1
%\p^c\paren{x_1-x_2\leq \underline{\underline{Y}}^1_N,x_2-x_3\leq \underline{\underline{Y}}^2_N}/(1-e^{-c})-1\\
%= \p^c\paren{-y< \underline{\underline{Y}}_N}/(1-e^{-c})-1
=\esp{\sum_1 e^{-cn}\indi{-y< \underline{\underline{Y}}_n}}. 
%\end{split}
\end{equation}For any $n\in \na$, it is known that $X$ and the time-reversed process $X^*$ has the same distribution, with
\begin{lesn}
	X^*_i=X_n-X_{n-i}\ \ \ \ \ \ \ \mbox{ for $0\leq i\leq n$.}
\end{lesn}This chain has components $X^*=(X^{1,*},X^{2,*},X^{3,*})$, and similarly for $Y^*$. 
Then
\begin{lesn}
	\left\{-y< \underline{\underline{Y}}_n \right\}\stackrel{d}{=}\left\{\max\{-Y^{1,*}_j,j\in [n]\}< x_2-x_1,\max\{-Y^{2,*}_j,j\in [n]\}< x_3-x_2 \right\}.
\end{lesn}Define $\overline{Y}^k_n=\max\{Y^k_j,0\leq j\leq n \}$ for $k=1,2$ and $\overline{Y}_n=(\overline{Y}^1_n,\overline{Y}^2_n)$. 
Add and subtract the term $Y^{k,*}_n$ and use $Y^{k,*}_i-Y^{k,*}_n=Y^k_n-Y^k_{n-i}-Y^{k,*}_n=-Y^k_{n-i}$ for $k=1,2$, so 
\begin{equation}\label{eqnEqualityInDnOfDoubleUnderlineYAndOverlineYMinusY}
%& \left\{\max_{i\in [n]}\{Y^1_{n-i}\}-Y^1_n\leq  x_2-x_1,\max_{i\in [n]}\{Y^2_{n-i}\}-Y^2_n\leq  x_3-x_2 \right\}\\
\left\{-y< \underline{\underline{Y}}_n \right\}\stackrel{d}{=} \left\{\overline{Y}_{n-1}-Y_n< y \right\}.
\end{equation}
This implies that $h^\uparrow_c$ can be reexpresed as
\begin{equation}\label{eqnFirstExpresionOfhc}
h^\uparrow_c(x)=\frac{1+\sum_1^\infty e^{-cn}\proba{\overline{Y}_{n-1}-Y_n< y}}{1+\sum_1^\infty e^{-cn}\proba{\overline{Y}_{n-1}-Y_n< 0 }}
\end{equation}

%\subsection{Partitioning $\na$ with regenerative intervals to obtain the limit of the $h_c^\uparrow$ transform}
\subsection[Partitioning $\na$ via the times of a multidimensional ladder height function to obtain the limit of the $h$-function]{Partitioning $\na$ via the times of a multidimensional ladder height function to obtain the limit of $h_c^\uparrow$}\label{subsectionORWPartitioningWithRegenerativeIntervals}

In this subsection, we partition $\na$ at some particular times $\{J_i,i\in\na \}$.
Those are the times in common among the 
%\correccion{weak? revisar si deben ser d\'ebiles}
ascending ladder times of $Y^1$ and $Y^2$, that is, if $(\alpha^k_j,j\geq 0)$ are the strict ascending ladder times of $Y^k$, then $J_i$ is the $i$th time such that $\alpha^1_j=\alpha^2_l$ for some $j,l\in \na$. 
We prove that the subpaths $\{Y_{J_i+n},0\leq n<J_{i+1}-J_i \}_i$ are i.i.d., and at the times $J_i$, every component of the walk $Y$ is at least as big as the current cumulative maximum.
In this sense, the reader should think on those subpaths as \emph{excursions} of $Y$.

%Let $\{\alpha^k_i,i\geq 0 \}$ be defined by $\alpha^k_0=0$, and for $i\geq 1$, define $\alpha^k_{i+1}=\inf\{n>\alpha^k_i:Y^k_n\geq \overline{Y}^k_{n-1} \}$, for $k=1,2$.

Let $J_0=0$ and for $i\in \na$, define
\begin{lesn}
	J_{i+1}=\min\left\{ n>J_i:\overline{Y}^k_{n-1}<Y^k_n,k=1,2 \right\},
\end{lesn}the first time after $J_i$, such that both walks reach the current maximum at the same time. 

\begin{remark}
	Note that $\overline{Y}_{J_i}=Y_{J_i}$, since both processes are at the same maximum. 
	Also, since we assumed $\proba{Y_1>0}>0$, then $\proba{J_1=1}=\proba{\overline{Y}_0<Y_1}=\proba{0<Y_1}$ has positive probability.
\end{remark}

We prove $(J_i,i\geq 0)$ are stopping times. 
Let $(\F_n,n\in \na)$ be the natural filtration of $X$. 
For any $m\in \na$, the event $\{J_1=m\}$ is equal to
%that $Y^1_n<\overline{Y}^1_{n-1}$ or $Y^2_n<\overline{Y}^2_{n-1}$ for any $n\in [m-1]$, and $Y^1_n\geq \overline{Y}^1_{n-1}$ and  $Y^2_n\geq \overline{Y}^2_{n-1}$. 
\begin{align*}
%	&\left\{J_1=m\right\}\\
\left\{\overline{Y}_{j-1}^k\geq Y_{j}^k,1\leq j\leq m-1  \mbox{ for $k=1$ or $k=2$}\right\}\cap \left\{\overline{Y}_{m-1}^k< Y_{m}^k,k=1,2\right\},
\end{align*}which is in $\F_m$.
%This clearly belongs to $\F_m$.
Assuming $J_i$ is a stopping time, the event $\{J_{i+1}=m \}$ is equal to
%$Y^1_{J_i+n}<\overline{Y}^1_{J_i+n-1}$ or $Y^2_{J_i+n}<\overline{Y}^2_{J_i+n-1}$ for any $n\in [m-J_i-1]$, and $Y^1_m\geq \overline{Y}^1_{m-1}$ and  $Y^2_m\geq \overline{Y}^2_{m-1}$. 
\begin{align*}
%&\left\{J_{i+1}=m\right\}\\
\bigcup_{l=i}^{m-1}\paren{\left\{J_i=l\right\}\cap \left\{\overline{Y}_{j-1}^k\geq Y_{j}^k,l+1\leq j\leq m-1  \mbox{ for $k=1$ or $k=2$}\right\}\cap \left\{\overline{Y}_{m-1}^k< Y_{m}^k,k=1,2\right\}},
\end{align*}which also belongs to $\F_m$.
%This also belongs to $\F_m$. 

We prove the independence and distribution between such times.

\begin{lemma}\label{lemmaYProcessBetweenJi}
	For every $i\in \na$, the walk $\{\overline{Y}_{J_i+n-1}-Y_{J_i+n},n\geq 1 \}$ is independent of $\F_{J_i}$ and has the same distribution as $\{\overline{Y}_{n-1}-Y_{n},n\geq 1 \}$. 
\end{lemma}
\begin{proof}
	Let $T<\infty$ be a stopping time. 
	For $n\geq 2$, decompose $\overline{Y}_{T+n-1}$ as the maximum up to time $T$ and the maximum between times $\{T+1,\ldots, T+n-1 \}$.
	Hence
	\begin{lesn}
		\overline{Y}_{T+n-1}-Y_{T+n}=\paren{\overline{Y}_T-Y_T}\vee \max\{Y_{T+l}-Y_T,l\in [n-1] \}-(Y_{T+n}-Y_T),
	\end{lesn}and for $n=1$
	\begin{lesn}
		\overline{Y}_T-Y_{T+1}=\paren{\overline{Y}_T-Y_T}\vee (0,0)-\paren{Y_{T+1}-Y_T}.
	\end{lesn}We substitute $T=J_i$ for $i\in \na$ and recall $\overline{Y}_{J_i}=Y_{J_i}$. 
	For $n\in\na$ and $A_m\in \re^2$ with $m\in [n]$, the events
	\begin{align*}
	& \bigcap_{m=1}^{n}\{ \overline{Y}_{J_i+m-1}-Y_{J_i+m}\in A_m\}\\
	& =\bigcap_{m=1}^{n}\left\{(0,0)\vee \max\{ Y_{J_i+l}-Y_{J_i},l\in [m-1]\}-(Y_{J_i+m}-Y_{J_i})\in A_m  \right\}
	\end{align*}are independent of $\F_{J_i}$ under $\{J_i<\infty \}$, by the strong Markov property.
	They also have the same distribution as
	\begin{lesn}
		\bigcap_{m=1}^{n}\left\{(0,0)\vee \max\{ Y_{l},l\in [m-1]\}-Y_{m}\in A_m  \right\} =\bigcap_{m=1}^{n}\left\{\overline{Y}_{m-1}-Y_{m}\in A_m  \right\},
	\end{lesn}recalling that $\overline{Y}_0=Y_0=(0,0)$ under $\p$.
\end{proof}

The following result is crucial to partition the sums in \eqref{eqnFirstExpresionOfhc}.

\begin{lemma}\label{lemmaJiAreIID}
	The times $\{J_{i+1}-J_{i},i\in \na\}$ are i.i.d. and $J_{i+1}-J_i=J_1\circ \theta_{J_i}$, where $\theta$ is the translation operator.
\end{lemma}
\begin{proof}
	For $i\in \na$ we have
	\begin{align*}
	J_{i+1}-J_i&=\min \{n>0: \overline{Y}_{J_i+n-1}< Y_{J_i+n}\}\\
	%&=\min \{n>0: \overline{Y}_{n-1+J_i}^k- Y_{n+J_i}^k\leq 0,k=1,2\}\\
	&=\min \{n>0: \max\{Y_{J_i+m}-Y_{J_i};0\leq m\leq n-1\}-(Y_{J_i+n}-Y_{J_i})<0\}\\
	%&=\min \{n>0: \max\{Y_{m}^k-Y_0^k;0\leq m\leq n-1\}-(Y_{n}^k-Y_0^k)\leq 0,k=1,2\}\circ \theta_{J_i}\\
	&=J_1\circ \theta_{J_i}.
	\end{align*}Then $J_{i+1}-J_i$ is independent of $\F_{J_i}$ and has the same law as $J_1$, by Lemma \ref{lemmaYProcessBetweenJi}.
\end{proof}

\begin{lemma}\label{lemmaConvergenceOfhUparrowSubc}
	The $h$-function $h^\uparrow_c$ converges as $c\downarrow 0$ to 
	\begin{lesn}
		h^{\uparrow}(x)= 1+\esp{\sum_{n=1}^{J_1-1}\indi{\overline{Y}_{n-1}-Y_{n}<y}},
	\end{lesn}recalling that $y=(x_2-x_1,x_3-x_2)$. 
\end{lemma}
\begin{proof}
	Recall Equation \eqref{eqnFirstExpresionOfhc}.
	Partition $\na$ at times $(J_i,i\in \na)$
	\begin{align*}
	& \esp{\sum_1e^{-cn}\indi{\overline{Y}_{n-1}-Y_n<y}}\\
	%& =1+\esp{\sum_0e^{-cJ_i}\indi{J_i<\infty}\sum_{n=J_i+1}^{J_{i+1}}e^{-c(n-J_i)}\indi{\overline{Y}_{n-1}^1-Y_n^1\leq x_2-x_1,\overline{Y}_{n-1}^2-Y_n^2\leq x_{3}-x_2}}\\
	& =\esp{\sum_0e^{-cJ_i}\indi{J_i<\infty}\sum_{n=1}^{J_{i+1}-J_i}e^{-cn}\indi{\overline{Y}_{n-1+J_i}-Y_{n+J_i}< y}}.
	\end{align*}Conditioning with $\F_{J_i}$ and summing over the values taken by $J_{i+1}-J_i$, the previous equation is equal to
	%\begin{align*}
	%	 =&\E\left\{\sum_0e^{-cJ_i}\indi{J_i<\infty}\sum_{m\geq 0} \sum_{n=1}^{m}e^{-cn}\right.\\
	%	& \left.\times \probac{J_{i+1}-J_i=m,\overline{Y}_{n-1+J_i}-Y_{n+J_i}\leq y}{\F_{J_i}}\right\}.	
	%\end{align*}
	\begin{lesn}
		\E\left\{\sum_0e^{-cJ_i}\indi{J_i<\infty}\sum_{m\geq 1} \sum_{n=1}^{m}e^{-cn} \probac{J_{i+1}-J_i=m,\overline{Y}_{n-1+J_i}-Y_{n+J_i}< y}{\F_{J_i}}\right\}.	
	\end{lesn}Using lemmas \ref{lemmaYProcessBetweenJi} and \ref{lemmaJiAreIID}, we obtain
	\begin{align*}
	&\p^c\paren{X_i+x\in \weyl,i\in [\zeta]}/(1-e^{-c})\\
	& =1+\esp{\sum^\infty_0e^{-cJ_i}\indi{J_i<\infty}}\esp{\sum_{n=1}^{J_1}e^{-cn}\indi{\overline{Y}_{n-1}-Y_{n}<y}}.
	\end{align*}Since $e^{-cJ_i}\indi{J_i=\infty}=0$, we can ignore the indicator $\indi{J_i<\infty}$.
	When $x=0$, the only term that remains in the second expectation above is $e^{-cJ_1}$, since $\overline{Y}_{n-1}^1\geq Y_{n}^1$ or $\overline{Y}_{n-1}^2\geq Y_{n}^2$ for $n<J_1$. 
	It follows that
	\begin{lesn}
		\p^c\paren{X_i\in \weyl,i\in [\zeta]}/(1-e^{-c})
		=1+\esp{\sum^\infty_0e^{-cJ_i}}\esp{e^{-cJ_1}}.
	\end{lesn}Dividing both terms, and using
	\begin{lesn}
		\esp{\sum_{n=1}^{J_1}e^{-cn}\indi{\overline{Y}_{n-1}-Y_{n}< y}}-\esp{e^{-cJ_1}} =\esp{\sum_{n=1}^{J_1-1}e^{-cn}\indi{\overline{Y}_{n-1}-Y_{n}< y }},
	\end{lesn}we have
	\begin{lesn}
		h_c^{\uparrow}(x)=1+\esp{\sum_{n=1}^{J_1-1}e^{-cn}\indi{\overline{Y}_{n-1}-Y_{n}<y}}\frac{\esp{\sum_0^\infty e^{-cJ_i}}}{1+\esp{\sum_0^\infty e^{-cJ_i}}\esp{e^{-cJ_1}}}.
	\end{lesn}Since $J_i=\sum_0^{i-1}(J_{k+1}-J_k)$ is a sum of i.i.d. random variables, then
	\begin{lesn}
		\esp{\sum_0^\infty e^{-cJ_i}}=\paren{1-\esp{e^{-cJ_1}}}^{-1},
	\end{lesn}implying
	\begin{lesn}
		h_c^{\uparrow}(x)=1+\esp{\sum_{n=1}^{J_1-1}e^{-cn}\indi{\overline{Y}_{n-1}-Y_{n}<y}}. 
	\end{lesn}The result follows from the monotone convergence theorem.
\end{proof}

The latter result implies Theorem \ref{teoOrderedRWIsAMarkovChainIntro}.
%\begin{teo}\label{teoOrderedRWIsAMarkovChainIntro}
%Let $x\in\weyl$. 
%Then, under $\p$, the random walk $X+x$ conditioned to stay ordered forever is a Markov chain with transition probabilities
%	\begin{esn}
%		p^{\uparrow}(w,dz)=\indi{z\in \weyl}\frac{h^{\uparrow}(z)}{h^{\uparrow}(w)}p(w,dz)
%	\end{esn}where $w\in \weyl$, 
%	\begin{esn}
%		p(w,dz)=\proba{X_1+w\in dz}
%	\end{esn}is the one-step transition probability of $X$, and $h^{\uparrow}$ is given by
%	%\correccion{Creo q deber\'ia empezar desde cero el $i$}
%	\begin{esn}
%		h^{\uparrow}(x)=1+\esp{\sum_{n=1}^{J_1-1}\indi{\overline{Y}_{n-1}-Y_{n}<y}},
%	\end{esn}with $J_1=\inf\{n>0: \overline{Y}_{n-1}^k< Y_{n}^k,k=1,2\}$.
%\end{teo}
In the next section, we prove that $h^\uparrow$ is (sub)harmonic, and give a simple condition that ensures it is finite.

\section[Properties of the limiting $h$-function and the interpretation of the walk as conditioned to stay ordered forever]{Properties of $h^\uparrow$ and the interpretation of the walk as conditioned to stay ordered forever}\label{sectionPropertiesOfhUparrow}

\subsection[The harmonicity of the $h$-function depends on the first exit time to \weyl]{The harmonicity of $h^\uparrow$ depends on $\esp{\tau}$}
We know that the first exit time from the Weyl chamber is given by
\begin{lesn}
	\tau=\min \{n>0:Y^1_n \wedge Y^2_n\leq 0\}.
\end{lesn}By Lemma \ref{lemmaConvergenceOfhUparrowSubc}, we rewrite $h^\uparrow$ as
\begin{lesn}
	h^\uparrow(x) =\lim_{c\to 0^+}\frac{\p^c_x\paren{\underline{\underline{Y}}_N> 0}}{\p^c\paren{\underline{\underline{Y}}_N>0}}=\lim_{c\to 0^+}\frac{\p_x\paren{\tau>N}}{\p\paren{\tau	>N}}.
\end{lesn}Let $Q_x$ be the law of $X$ killed at the first exit of the Weyl chamber, that is, for $n\in \na$ and $\Lambda \in \F_n$
\begin{lesn}
	Q_x\paren{\Lambda, n<\zeta}=\p_x\paren{\Lambda,n<\tau }.
\end{lesn}Expectations under $Q_x$ will be denoted by $\E^Q_x$. The next lemma gives us conditions to know if $h^\uparrow$ is harmonic or subharmonic. 
It is based on Lemma 1 of \cite{MR2164035}. 

\begin{lemma}\label{lemmahUparrowIsSubharmonicOrHarmonic}
	Let $x\in \weyl$. 
	If $\esp{\tau}<\infty$, then $h^\uparrow$ is subharmonic and 
	\begin{lesn}
		\E^Q_x\paren{h^\uparrow(X_n)\indi{n<\zeta}}< h^\uparrow(x).
	\end{lesn}If $\esp{\tau}=\infty$, then $h^\uparrow$ is harmonic and 
	\begin{lesn}
		\E^Q_x\paren{h^\uparrow(X_n)\indi{n<\zeta}}= h^\uparrow(x).
	\end{lesn}
\end{lemma}
\begin{proof}
	Since we proved in Lemma \ref{lemmaConvergenceOfhUparrowSubc} that the convergence of $h^\uparrow_c$ to $h^\uparrow$ is monotone, then
	\begin{equation}\label{eqnh_subharmonic}
	\E^Q_x\paren{h^\uparrow(X_n)\indi{n<\zeta}}=\lim_{c\to 0^+}\E_x\paren{\frac{\p_{X_n}\paren{\tau>N}}{\p\paren{\tau>N}}\indi{n<\tau}}.
	\end{equation}Using the Markov property
	\begin{align*}
	\p_x\paren{\tau>n+N}& =\E_x\paren{\indi{Y^2_k\wedge Y^1_k>0,k\in [n+N]}}\\
	& =\E_x\paren{\indi{Y^2_k\wedge Y^1_k>0,k\in[n]}\indi{Y^2_k\wedge Y^1_k>0,k\in[N]}\circ \theta_n}\\
	& =\E_x\paren{\indi{\tau> n}\p_{X_n}\paren{\tau> N}},
	\end{align*}which is the numerator in the right-hand side of Equation \eqref{eqnh_subharmonic}.
	Summing over all the values of $N$
	\begin{align*}
	\p_x\paren{\tau>n+N}& =\sum_k \p_x\paren{\tau>n+k,N=k}\\
	%& =\sum_m \p_{x^{\prime}}\paren{\tau>n+m}(1-e^{-c})e^{-cm}\\
	& =e^{cn}\sum_{k\geq n} \p_x\paren{\tau>k}(1-e^{-c})e^{-ck}.
	\end{align*}Starting the sum from $k=0$, we obtain
	\begin{lesn}
		\p_x\paren{\tau>n+N} =e^{cn}\left\{\p_x\paren{\tau>N}-\sum_{0}^{n-1}\p_x\paren{\tau>k}\proba{N=k}\right\}.
	\end{lesn}Thus, the right-hand side of Equation \eqref{eqnh_subharmonic} is equal to
	\begin{align*}
	& \lim_{c\to 0^+}e^{cn}\left\{\frac{\p_x\paren{\tau>N}}{\p\paren{\tau>N}}-\sum_{0}^{n-1} \frac{\p_x\paren{\tau>k}e^{-ck}}{\sum \p\paren{\tau>m}e^{-cm}}\right\}\\
	& =h^\uparrow(x)-\frac{1}{\E(\tau)}\sum_{0}^{n-1} \p_x\paren{\tau>k},
	\end{align*}which proves the lemma, since $\p_x(\tau>0)=\p(x+X_0\in\weyl )=1$.
\end{proof}

\subsection[Finiteness of the $h$-function]{Finiteness of $h^\uparrow$}\label{subsectionORWFinitenessOfhUparrow}

To prove $h^\uparrow(x)<\infty$ for every $x\in \weyl$, we use the remark of Lemma 1 in \cite{MR1001739}. 
In this subsection, the inequality $x>z$ for $x,z\in \re^3$ means there is strict inequality component-wise.

\begin{lemma}\label{lemmahUparrowIsFinite}
	Assume there exists $\epsilon=(\epsilon_1,\epsilon_2)\in \re_+$ such that
	\begin{lesn}
		\proba{(X^2_1-X^1_1,X^3_1-X^2_1)>\epsilon}>0.
	\end{lesn}Then
	\begin{lesn}
		h^\uparrow(x)<\infty\ \ \ \ \ \ \ \ \ \ \ \forall\, x\in \weyl.
	\end{lesn}
\end{lemma}
\begin{proof}
	Note that Lemma \ref{lemmahUparrowIsSubharmonicOrHarmonic} was independent of the finiteness of $h^\uparrow$. 
	Hence, from such lemma and $x\in \weyl$ we have
	\begin{lesn}
		h^\uparrow(x)\geq \int \proba{x+X_1\in dz,1<\tau}h^\uparrow(z)=\int_{z\in \weyl} \proba{x+X_1\in dz}h^\uparrow(z).
	\end{lesn}Define $g(x)=(x_2-x_1,x_3-x_2)$ for $x\in\re_+^3$. 
	For simplicity, instead of $g(x)$ we write $x^-$. 
	So, for instance $X^-=(X^2-X^1,X^3-X^2)$ and $x^-:=(x_2-x_1,x_3-x_2)$. 
	Then, we have
	\begin{lesn}
		\proba{x+X_1\in dz}=\proba{x_1+X^1_1\in dz_1,x^-+X^-_1\in dz^-}.
	\end{lesn}Note from Lemma \eqref{lemmaConvergenceOfhUparrowSubc} that $h^\uparrow(x)$ depends on $x$ only trough $x^-$.
	Define $h^-:\re^2_+\cup\{ (0,0)\}\mapsto \re_+$ as $h^-(x^-):=h^\uparrow(x)$, so
	\begin{lesn}
		h^-(x^-)=1+\esp{\sum_{n=1}^{J_1-1}\indi{\overline{Y}_{n-1}-Y_{n}<x^-}}.
	\end{lesn}It follows that
	\begin{equation}\label{eqnInequalityhUparrowAndOneStepTransition}
	h^-(x^-)\geq \int_{z\in \weyl} \proba{x_1+X^1_1\in dz_1,x^-+X^-_1\in dz^-}h^-(z^-).
	\end{equation}Also, note that $h^-(0)=1$, since $\indi{\overline{X}^-_0-X^-_1\leq 0}=1$ implies $J_1=1$. 
	
	Assume that $h^-(z^-)=\infty$ for every $z^->\epsilon$. 
	Fix any $x_1\in \re$, and use $x=(x_1,x_1,x_1)$ in \eqref{eqnInequalityhUparrowAndOneStepTransition} to obtain
	\begin{lesn}
		1=h^-(0)=h^\uparrow(x)\geq \int_{z\in \weyl\cap\{z\in \re^3:z^->\epsilon \}} \proba{x_1+X^1_1\in dz_1,x^-+X^-_1\in dz^-}h^-(z^-).
	\end{lesn}Hence, it should be the case that
	\begin{lesn}
		0=\proba{x_1+X^1_1\in\re, x^-+X^-_1>\epsilon }=\proba{X^-_1>\epsilon },
	\end{lesn}contradicting the hypothesis.
	Therefore, there exists $z_{(1)}=(x_1,x_1+z_{(1),1},x_1+z_{(1),1}+z_{(1),2})\in \weyl$ such that
	\begin{lesn}
		z^-_{(1)}>\epsilon\ \ \ \ \ \mbox{ and }\ \ \ \ \ h^\uparrow(z_{(1)})=h^-(z^-_{(1)})<\infty.
	\end{lesn}Now, assume $h^-(z^-)=\infty$ for every $z^->\epsilon+z^-_{(1)}$. 
	Use $x=z^-_{(1)}$ in \eqref{eqnInequalityhUparrowAndOneStepTransition} to obtain
	\begin{lesn}
		\infty>h^-(z^-_{(1)})\geq \int_{z\in \weyl\cap\{z\in \re^3:z^->\epsilon +z^-_{(1)}\}} \proba{x_1+X^1_1\in dz_1,z^-_{(1)}+X^-_1\in dz^-}h^-(z^-).
	\end{lesn}Then, it should happen that
	\begin{lesn}
		0=\proba{x_1+X^1_1\in\re, z^-_{(1)}+X^-_1>\epsilon+z^-_{(1)} }=\proba{X^-_1>\epsilon },
	\end{lesn}again contradicting the hypothesis. 
	Hence, there exists $z_{(2)}=(x_1,x_1+z_{(2),1},x_1+z_{(2),1}+z_{(2),2})\in \weyl$ such that
	\begin{lesn}
		z^-_{(2)}>\epsilon+z^-_{(1)}\ \ \ \ \ \mbox{ and }\ \ \ \ \ h^\uparrow(z_{(2)})=h^-(z^-_{(2)})<\infty.
	\end{lesn}
	Continuing in this way, there is some subsequence $(z_{(n)},n\in \na)$, with $z_{(n)}=(x_1,x_1+z_{(n),1},x_1+z_{(n),1}+z_{(n),2})\in \weyl$ satisfying
	\begin{lesn}
		z^-_{(n)}>\epsilon+z^-_{(n-1)}\ \ \ \ \ \mbox{ and }\ \ \ \ \ h^\uparrow(z_{(n)})=h^-(z^-_{(n)})<\infty,
	\end{lesn}for every $n$. 
	
	Fix any $x=(x_1,x_2,x_3)\in\weyl$. 
	We prove that $h^\uparrow(x)<\infty$. 
	Note that in the previous analysis, $x_1$ was arbitrary. 
	Let $n\in \na$ such that
	\begin{lesn}
		z^-_{(n),1}\wedge z^-_{(n),2}>(n\epsilon_1)\wedge (n\epsilon_2)>(x_3-x_2)\vee (x_2-x_1).
	\end{lesn}It follows that
	\begin{align*}
	h^{\uparrow}(x) & =1+\esp{\sum_{n=1}^{J_1-1}\indi{\overline{X}^-_{n-1}-X^-_{n}<x^-}}\\
	&\leq  1+\esp{\sum_{n=1}^{J_1-1}\indi{\overline{X}^-_{n-1}-X^-_{n}< z^-_{(n)}}},
	\end{align*}which is finite by construction.
\end{proof}

\subsection{Ordered random walks as the limit law of random walks conditioned to stay ordered up to a geometric time}

Let $(q_n,n\geq 1)$ be the transition probabilities of $(X,Q)$.
From Theorem \ref{teoOrderedRWIsAMarkovChainIntro}, denote by $(p^{\uparrow}_n,n\geq 1)$ the transition probabilities of $X$ conditioned to stay ordered
\begin{esn}
	p^\uparrow_n(w,dz)=\frac{h^{\uparrow}(z)}{h^{\uparrow}(w)}q_n(w,dz)\ \ \ \ \ \ \ \ \ w\in \weyl,\ n\in \na.
\end{esn}The law of the Markov process with transition probabilities $(p^{\uparrow}_n,n\geq 1)$ and starting from $x\in \weyl$ is denoted by $\p^\uparrow_x$.
Hence, for $n\in \na$ and $\Lambda\in \F_n$
\begin{equation}\label{eqnDefinitionOfPUparrow}
\p^\uparrow_x(\Lambda ,n<\zeta)=\frac{1}{h^\uparrow(x)}\E^Q_x\paren{h^\uparrow(X_n)\indi{\Lambda,n<\zeta}}.
\end{equation}Its lifetime is $\p^\uparrow_x$-finite if $h^\uparrow$ is subharmonic and $\p^\uparrow_x$-infinite if it is harmonic.
Let us prove $(X,\p^\uparrow_x)$ is the limit as $c\to 0+$ of $(X,\p_x)$ conditioned to have ordered components up to a geometric time. 

\begin{lemma}\label{lemmaRWUnderHTransformIsTheLimitOFConditioningUntilExpTime}
	Let $N$ be geometric time with parameter $1-e^{-c}$, independent of $(X,\p)$. 
	Then, for every $x\in \weyl$, every finite $\F_n$-stopping time $T$ and $\Lambda\in \F_T$
	\begin{lesn}
		\lim_{c\to 0^+}\p_x\paren{\Lambda,T\leq N|X(i)\in \weyl, i\in [N]}=\p^\uparrow_x(\Lambda, T<\zeta).
	\end{lesn}
\end{lemma}
\begin{proof}
	First we use a deterministic time $T\in \na$. 
	Note that $\{T<\tau \}=\{X(i)\in \weyl, i\in [T]\}$.
	We work with $\p_x\paren{\Lambda,T\leq N,X(i)\in \weyl, i\in [N]}$. 
	Separating in the first $T$ values of $X$ and summing over all the values of $N$
	\begin{align*}
	& \p_x\paren{\Lambda,T\leq N,X(i)\in \weyl, i\in [N]}\\
	& = \sum_{n\geq T}\p_x\paren{\Lambda,T\leq n,X(i)\in \weyl, i\in [T],X(T+i)\in \weyl, i\in [n-T]}\proba{N=n},
	\end{align*}starting the sum at zero and using the Markov property at $\F_T$
	\begin{align*}
	& \p_x\paren{\Lambda,T\leq N,X(i)\in \weyl, i\in [N]}\\
	& = e^{-cT}\sum_{n\geq 0}\p_x\paren{\Lambda,T<\tau,\p_x\paren{X(T+i)\in \weyl, i\in [n]|\F_T}}\proba{N=n}\\
	& = e^{-cT}\p_x\paren{\Lambda,T<\tau,\p_{X_T}\paren{N<\tau}}\\
	& = \p_x\paren{\Lambda,T<\tau,t\leq N,\p_{X_T}\paren{N<\tau}}.
	\end{align*}Now, consider $c_0>0$ and any $c\in (0,c_0)$. 
	Recall from Lemma \ref{lemmaConvergenceOfhUparrowSubc} that $h^\uparrow_c$ increases to $h^\uparrow$, hence
	\begin{align*}
	\indi{\Lambda,T<\tau,T\leq N }\frac{\p_{X_T}\paren{\tau >N}}{\p_x\paren{\tau >N}}&= \indi{\Lambda,T<\tau,T\leq N}\frac{h^\uparrow_c(X_T)}{h^\uparrow_c(x)}\\
	&\leq  \indi{\Lambda,T<\tau}\frac{h^\uparrow(X_T)}{h^\uparrow_{c_0}(x)}. 
	\end{align*}Taking expectations on both sides and using Lemma \ref{lemmahUparrowIsSubharmonicOrHarmonic}
	\begin{lesn}
		\E_x\paren{\indi{\Lambda,T<\tau,T\leq N}\frac{\p_{X_T}\paren{\tau >N}}{\p_x\paren{\tau >N}}}\leq \frac{h^\uparrow(x)}{h^\uparrow_{c_0}(x)},
	\end{lesn}and the right-hand side is finite by Lemma \ref{lemmahUparrowIsFinite}. 
	Hence, by Lebesgue's dominated convergence theorem
	\begin{lesn}
		\lim_{c\to 0^+}\p_x\paren{\Lambda,T\leq N|\tau>N}=\lim_{c\to 0^+}
		\E_x\paren{\indi{\Lambda,T<\tau,T\leq N }\frac{\p_{X_T}\paren{\tau >N}}{\p_x\paren{\tau >N}}}=\p^\uparrow_x\paren{\Lambda, T<\zeta}.
	\end{lesn}
	
	Let us prove the same convergence for any finite stopping time $T$.
	Summing over all the values of $T$, the equality
	\begin{lesn}
		\p_x\paren{\Lambda,T\leq N<\tau}=\p_x\paren{\Lambda,T<\tau,t\leq N,\p_{X_T}\paren{N<\tau}}
	\end{lesn}and Equation \eqref{eqnDefinitionOfPUparrow} holds for $T$.
	We need to prove Lemma \ref{lemmahUparrowIsSubharmonicOrHarmonic} holds true for any stopping time $T<\infty$ a.s.
	Summing over all the values of $T$, in the subharmonic case
	\begin{align*}
	\E^Q_x\paren{h^\uparrow(X_T)\indi{T<\zeta}}& =\sum_n\E_x\paren{h^{\uparrow}(X_n)\indi{n<\tau,T=n}}\\
	& \leq  h^\uparrow(x)\sum_n \p^\uparrow_x\paren{T=n, n<\zeta}\\
	& =  h^\uparrow(x)\p^\uparrow_x\paren{T<\infty, T<\zeta},
	\end{align*}which is smaller than $h^\uparrow(x)$. 
	In the harmonic case, the inequality above is an equality, so it remains to prove that $\p^\uparrow_x(T<\infty, T<\zeta)=1$. 
	But this is clear since
	\begin{lesn}
		\p^\uparrow_x(T=\infty,T<\zeta)=\lim_n\frac{1}{h^\uparrow(x)}\E_x\paren{h^{\uparrow}(X_T)\indi{T>n}}=0,
	\end{lesn}by monotone convergence.
	%\begin{align*}
	%	\p^\uparrow_x\paren{T<\infty}& = \frac{1}{h^\uparrow(x)}\E_x\paren{h^{\uparrow}(X_T)\indi{T<\infty}}\\
	%	& = \frac{1}{h^\uparrow(x)}\E_x\paren{h^{\uparrow}(X_T)}\\
	%	& = \sum \frac{1}{h^\uparrow(x)}\E_x\paren{h^{\uparrow}(X_n)\indi{T=n}}\\
	%	&= \sum_n \p^\uparrow_x\paren{T=n}\\
	%	&=\correccion{ 1.}
	%\end{align*}
\end{proof}

In the next section, we obtain several reexpresions of the $h$-function.

%%%%%%%%%%%%%%%%%%%%%%%%%%%%%%%%%%%%%%%%%%%%%%%%%%%%%%%%%%%%%%
\section[Reexpressions of the $h$-function]{Reexpressions of $h^\uparrow$}\label{sectionORWReexpressionsOfhUparrow}

\subsection{Reexpresions using the minimum of the descending ladder times of the components}

Changing the measure to start at zero, we have
\begin{lesn}
	\esp{\sum_1^\infty e^{-cn}\indi{-y< \underline{\underline{Y}}_n}}=\E_x\paren{\sum_1^\infty e^{-cn}\indi{0<\underline{\underline{Y}}_n}}.
\end{lesn}For $k=1,2$, denote by $(\beta^k_i,i\in \na)$ the strict descending ladder times of $Y^k$, that is $\beta^k_0=0$ and for $i\in \na$ the time $\beta^k_i$ is the smallest index $n$ such that $Y^k(\beta^k_{i-1}+n)< Y^k(\beta^k_{i-1})$.
The above sum stops when one component $Y^k$ becomes negative, that is, at $\beta^1_1\wedge \beta^2_1$.
Then
\begin{align*}
h_c^{\uparrow}(x)&=\frac{1+\E_x\paren{\sum_1^\infty e^{-cn}\indi{0<\underline{\underline{Y}}_n}}}{1+\E\paren{\sum_1^\infty e^{-cn}\indi{0<\underline{\underline{Y}}_n}}}\\
&=\frac{1+\E_x\paren{\sum_1^{\beta_1^1\wedge \beta_1^2-1}e^{-cn}}}{1+\E\paren{\sum_1^{\beta_1^1\wedge \beta_1^2-1}e^{-cn}}}.
\end{align*}This equality allows us to prove the next proposition.

\begin{lemma}\label{lemmahUpArrowTransformWhenSomeComponentDriftsToMinusInftyOrAllToInfty}
	If some component of $Y$ drifts to $-\infty$, the $h$-function $h^\uparrow$ is given by
	\begin{lesn}
		h^{\uparrow}(x)=\frac{\E_x\paren{\beta_1^1\wedge \beta_1^2}}{\E\paren{\beta_1^1\wedge \beta_1^2}}=\frac{\E_x\paren{\tau}}{\E\paren{\tau}}.
	\end{lesn}If every component drifts to $+\infty$ and 
	%$\proba{Y_1>0}>0$
	$\p\paren{\beta_1^1\wedge \beta_1^2=\infty}>0$, then
	\begin{lesn}
		h^{\uparrow}(x)=\frac{\p_x\paren{\beta_1^1\wedge \beta_1^2=\infty}}{\p\paren{\beta_1^1\wedge \beta_1^2=\infty}}=\frac{\p_x\paren{\tau=\infty}}{\p\paren{\tau=\infty}}.
	\end{lesn}
\end{lemma}
\begin{proof}
	If $Y^k$ drifts to $-\infty$ for some $k=1,2$, then $\esp{\beta^1_1\wedge \beta^2_1}\leq \esp{\beta^k_1}<\infty$ by Proposition 9.3, page 167 of \cite{MR1876169}. 
	Therefore, by the monotone convergence theorem
	\begin{lesn}
		h_c^{\uparrow}(x)\to \frac{1+\E_{x}\paren{\beta_1^1\wedge \beta_1^2-1}}{1+\E\paren{\beta_1^1\wedge \beta_1^2-1}}.
	\end{lesn}
	
	If every component drifts to $+\infty$, then $Y^1\wedge Y^2$ has a finite minimum with positive probability. 
	By hypothesis $\proba{\beta^1_1\wedge \beta^2_1=\infty}>0$, so 
	\begin{align*}
	h_c^{\uparrow}(x)&=\frac{\E_x\paren{\sum_0^{\beta_1^1\wedge \beta_1^2-1}e^{-cn}}}{\E\paren{\sum_0^{\beta_1^1\wedge \beta_1^2-1}e^{-cn}}}\\
	&=\frac{\E_x\paren{\paren{1-e^{-c(\beta_1^1\wedge \beta_1^2)}}\indi{\beta_1^1\wedge \beta_1^2<\infty}+\indi{\beta_1^1\wedge \beta_1^2=\infty}}}{\E\paren{\paren{1-e^{-c(\beta_1^1\wedge \beta_1^2)}}\indi{\beta_1^1\wedge \beta_1^2<\infty}+\indi{\beta_1^1\wedge \beta_1^2=\infty}}}\\
	&\to \frac{\p_x\paren{\beta_1^1\wedge \beta_1^2=\infty}}{\p\paren{\beta_1^1\wedge \beta_1^2=\infty}}.
	\end{align*}
\end{proof}

\subsection{Reexpresions using the union of the descending ladder times}

Let $\{\beta_1,\beta_2,\ldots \}$ be the ordered union of the positive strict descending ladder times of $Y^1$ and $Y^2$, that is, the ordered union of  $\{\beta^1_i,\beta^2_j,i,j\geq 1 \}$. 
Define $\beta_0=1$.
Denoting by $g_n=\beta_n$ and $d_n=\beta_{n+1}$ for $n\geq 0$, the set $\{g_n,g_n+1,\ldots, d_n-1 \}$ is the $n$th interval where $\underline{\underline{Y}}$ remains constant. 
Partitioning $\na$ on such intervals, from Equation \eqref{eqnhcUsingTheMinimum}
\begin{align*}
& 1+\E\paren{\sum_1^\infty e^{-cn}\indi{-y< \underline{\underline{Y}}_n}}\\
& = \esp{\sum_{n\geq 0}\indi{g_n<\infty}\sum_{k=g_n}^{d_n-1}e^{-c(k-g_n)}e^{-cg_n}\indi{-y<\underline{\underline{Y}}_k}} \\
%& = \esp{\sum_{n=0}e^{-cg_n}\indi{g_n<\infty}\indi{x_1-x_2\leq \underline{Y}^1_{g_n},x_2-x_3\leq \underline{Y^2}_{g_n}}\sum_{k=g_n}^{d_n-1}e^{-c(k-g_n)}}\\
& = \esp{\sum_{n\geq 0}e^{-cg_n}\indi{g_n<\infty,-y< \underline{\underline{Y}}_{g_n}}\sum_{k=0}^{d_n-g_n-1}e^{-ck}}.
\end{align*}The above equation for $x=0$ is
\begin{lesn}
	\E\paren{\sum_{n\geq 0} e^{-cn}\indi{0< \underline{\underline{Y}}_n}}=\esp{\sum_{k=0}^{d_0-1}e^{-ck}}.
\end{lesn}Note that $d_0=\beta^1_1\wedge \beta^2_1$.
Also, note that $-y\leq 0\leq \underline{\underline{Y}}_{d_0-1}$, therefore 
\begin{align*}
h_c^{\uparrow}(x) & = \frac{\esp{\sum_0^{d_0-1}e^{-ck}}+ \sum_{n\geq 1}\esp{e^{-cg_n}\indi{g_n<\infty,-y< \underline{\underline{Y}}_{g_n}}\sum_{k=0}^{d_n-g_n-1}e^{-ck}}}{\esp{\sum_0^{d_0-1}e^{-ck}}}\\ &=1+\paren{\esp{\sum_0^{d_0-1}e^{-ck}}}^{-1}\sum_{n\geq 1}\esp{e^{-cg_n}\indi{g_n<\infty,-y< \underline{\underline{Y}}_{g_n}}\sum_{k=0}^{d_n-g_n-1}e^{-ck}}.
\end{align*}As before, depending on the asymptotic behavior of the components, we can obtain a limit.

\begin{propo}
	If some component of $Y$ drifts to $-\infty$, then
	\begin{lesn}
		\esp{d_n-g_n}<\infty\ \ \ \ \ \ \forall n,
	\end{lesn}and the $h$-function is
	\begin{lesn}
		h^\uparrow(x)=1+\sum_{n\geq 1}\frac{\esp{d_n-g_n;g_n<\infty,-y< \underline{\underline{Y}}_{g_n}}}{\esp{d_0}}.
	\end{lesn}If every component drifts to $+\infty$ and $\proba{d_0=\infty}>0$, then
	\begin{lesn}
		h^\uparrow(x)=1+\sum_{n\geq 1}\frac{\proba{g_n<\infty,-y< \underline{\underline{Y}}_{g_n},d_n-g_n=\infty}}{\proba{d_0=\infty}}.
	\end{lesn}
\end{propo}
\begin{proof}
	As before, if the component $k$ drifts to $-\infty$, the first case follows by monotone convergence theorem and $d_n-g_n\leq \beta^k_j-\beta^k_{j-1}$ for some $j$. 
	The second case follows by
	\begin{lesn}
		\sum_0^{d_n-g_n-1}e^{-ck}=\frac{1-e^{-c(d_n-g_n)}}{1-e^{-c}}\indi{d_n<\infty}+\frac{1}{1-e^{-c}}\indi{d_n=\infty},
	\end{lesn}and using monotone convergence theorem.
\end{proof}

\subsection{Reexpresion as a renovation function}
Recall from the previous section that $(\beta_n,n\geq 0)$ is the ordered union of the strict descending ladder times of $Y^1$ and $Y^2$. 
We have the following result.

\begin{propo}\label{propoHTransformAsARenovationFunction}
	The $h$-function $h\uparrow$ can be expressed as 
	\begin{lesn}
		h^{\uparrow}(x)=1+\sum_{n=1}^{\infty}\proba{-\underline{\underline{Y}}_{\beta_n}<y}.
	\end{lesn}
\end{propo}
\begin{proof}
	First we express $h^\uparrow$ as an infinite sum, using Tonelli's theorem and Theorem \ref{teoOrderedRWIsAMarkovChainIntro}, we have
	\begin{equation}\label{eqnhcWithJ1GreaterThanN}
	%	\begin{split}
	h^{\uparrow}(x)-1=\esp{\sum_{n=1}^{J_1-1}\indi{\overline{Y}_{n-1}-Y_n<y}}
	=\sum_{n=1}^{\infty}\esp{\indi{\overline{Y}_{n-1}-Y_n<y}\indi{J_1>n}}.
	%	\end{split}
	\end{equation}The event $\{J_1>n \}$ means that for every $j\in [n]$, there is some $k$, such that the running maximum at time $j-1$ of $Y^k$ is at least $Y^k_j$.
	This is written as
	\begin{lesn}
		\{J_1>n \}=\bigcap_{j=1}^n\bigcup_k\left\{\max\left\{Y_l^{k};0\leq l\leq j-1\right\}-Y_{j}^{k}\geq 0\right\}. 
	\end{lesn}Recall the equality in distribution between $Y$ and $Y^*$, which is $Y$ reversed in time.
	Also, recall the equality in distribution of $-\underline{\underline{Y}}$ and $\overline{Y}_{\cdot-1}-Y_\cdot$ of Equation \eqref{eqnEqualityInDnOfDoubleUnderlineYAndOverlineYMinusY}.
	Hence, we have
	\begin{align*}
	\{ J_1>n\}& \stackrel{d}{=} \bigcap_{j=1}^n\bigcup_k\left\{\max\left\{Y_l^{k,*};0\leq l\leq j-1\right\}-Y_{j}^{k,*}>0\right\}\\
	& = \bigcap_{j=1}^n\bigcup_k\left\{\max\left\{-Y_{n-l}^{k};0\leq l\leq j-1\right\}+Y_{n-j}^{k}>0\right\}\\
	& =  \bigcap_{j=1}^n\bigcup_k\left\{\min\left\{Y_{n-l}^{k};0\leq l\leq j-1\right\}<Y_{n-j}^{k}\right\}.
	\end{align*}In a similar way, we can prove that
	\begin{lesn}
		\left\{\overline{Y}_{n-1}-Y_n\leq y,J_1>n \right\}\stackrel{d}{=}	\left\{\underline{\underline{Y}}^k_n\geq -y_k,\forall k\right\}\bigcap \bigcap_{j=1}^n\bigcup_k\left\{\min\left\{Y_{n-l}^{k};0\leq l\leq j-1\right\}<Y_{n-j}^{k}\right\}.
	\end{lesn}
	Now we prove the last term means $n$ is a strict descending ladder time of some $Y^k$. 
	In fact, reordering the index set, the last term is equal to
	\begin{align*}
	%&\bigcap_{j=1}^n\bigcup_k\left\{\min\left\{Y_{n-l}^{k};0\leq l\leq j-1\right\}<Y_{n-j}^{k}\right\}\\
	& \bigcap_{j=1}^n\bigcup_k\left\{\min\left\{Y_l^k;n-j+1\leq l\leq n\right\}<Y_{n-j}^{k}\right\}\\
	&= \bigcap_{j=1}^n\bigcup_k\left\{\min\left\{Y_l^k;j\leq l\leq n\right\}<Y_{j-1}^{k}\right\}.
	\end{align*}The right-hand side means the future minimum of $Y^k$ up to time $n$ is always smaller than the current value of $Y^k$, for some $k$. 
	Thus, the time $n$ is a strict descending ladder time of some $Y^k$, and
	\begin{lesn}
		h^{\uparrow}(x)=1+\sum_{n=1}^{\infty}\proba{\underline{\underline{Y}}_{\beta_n}>-y}.\qedhere
	\end{lesn}
\end{proof}

The next section is devoted to obtain conditions for the finiteness of $\esp{\tau}$.

\section[Known results about the expectation of the first exit time of \weyl, to ensure the $h$-function is harmonic]{Known results about the expectation of $\tau$ to ensure $h^\uparrow$ is harmonic}\label{sectionORWExpectationOfTau}

In Theorem 1 of \cite{MR2609589}, the tail of the distribution of $\tau$ is computed. 
Explicitly, let $X=(X^1,\ldots, X^d)$ be a random walk with i.i.d. components on $\re$. 
Under the assumptions that the step distribution has mean zero and the $\alpha$ moment is finite for $\alpha=d-1$ if $d>3$, and $\alpha>2$ if $d=3$, they prove 
\begin{lesn}
	\lim_nn^{d(d-1)/4}\p_x\paren{\tau>n}=KV(x), 
\end{lesn}where $K$ is an explicit constant and $V$ is given by
\begin{lesn}
	V(x)=\Delta(x)-\E_{x}(\Delta(X(\tau)))\ \ \ \ \ x\in \weyl\cap S^d,
\end{lesn}with $S\subset \re$ the state space of the random walks, and $\Delta$ defined in \eqref{eqnVandermondeDeterminant}.
This implies that for $x\in \weyl\cap S^d$
\begin{lesn}
	\E_x(\tau)<\infty\ \ \ \ \ \ \ \ \ \ \mbox{whenever }d\geq 3.
\end{lesn}This suggests that $\esp{\tau}<\infty$ in this case.
%\correccion{Empezando en cero? $\E_x(\tau)\leq \esp{\tau}$. }

%Let $h^{(d-1)}$ be the harmonic function constructed in \cite{MR2430709}, that is, the function in \eqref{eqnhTransformOfKonigEichelsbacher}, but for the first $d-1$ components of $X$.
%Denote by $v_1$ such function defined on the first $d-1$ components of $x=(x_1,\ldots, x_d)\in \weyl$, and $v_2$ on the last $d-1$ components. 
%If $v$ is any linear combination of $v_1$ and $v_2$, it is proved in \cite{MR2878783} that
%\begin{lesn}
%	\proba{\tau_x>n}\sim c U(x)n^{-\alpha/2-(k-1)(k-2)/4}\ \ \ \ \ \ \ \ n\to \infty,
%\end{lesn}with $c$ an absolute constant and $\alpha\in (k-2,k-1)$.
%Hence, also in this case $\esp{\tau_x}=\infty $ iff $d\leq 2$. 

In the paper \cite{MR3283611} the author obtains the asymptotic behavior of $\tau_x$, for random walks with non-zero drift killed when leaving general cones on $\re^d$.
Under some assumptions, in particular, the step distribution having all moments and a drift pointing out of the cone, it is proved the existence of a function $U$ such that
\begin{lesn}
	\p\paren{\tau_x>n}\sim \rho c^nn^{-p-d/2}U(x).
\end{lesn}The value $p\geq 1$ is the order of some homogeneous function, and $c\in [0,1]$.
This suggests $\esp{\tau_x}\leq \rho U(x)\sum  c^n<\infty$ whenever $c\in (0,1)$.

In the paper \cite{MR3342657}, the authors obtain
\begin{lesn}
	\proba{\tau_x>n}\sim cV(x)n^{-p/2}\ \ \ \ \ \ \ \ n\to \infty,
\end{lesn}for random walks in a cone, with components having zero mean, variance one, covariance zero, and some finite moment. 
In that case, the value $p$ is
\begin{lesn}
	p=\sqrt{\lambda_1+(d/2-1)^2}-(d/2-1)>0.
\end{lesn}Thus, the expectation of $\tau_x$ is infinite iff
\begin{align*}
1\geq p/2& \iff \paren{d/2+1}^2\geq \lambda_1+(d/2-1)^2\\
& \iff 2d\geq \lambda_1.
\end{align*}

The paper \cite{MR3512425}, computes the asymptotic exit time probability for random walks in cones, under some general conditions. 
The first is that the support of the probability measure of $X(1)$ is not included in any linear hyperplane. 
The second is that, if $L$ is the Laplace transform of the random walk having $x^*$ as a minimum, then $L$ is finite on an open neighborhood of $x^*$, and that this value belongs to the dual cone. 
Under such hypotheses, they prove that
\begin{lesn}
	\lim_{n\to\infty}\p_x\paren{\tau>n}^{1/n}=L(x^*), 
\end{lesn}for all $x\in K_\delta:=K+\delta v$, for some $\delta\geq 0$ and some fixed $v$ in $K^o$.
The authors note that in general, there is no explicit link between the drift $m$ of the walk (if exists), $x^*$ and $L(x^*)$. 
The only exception is when $m\in K$.
In such case, $L(x^*)=1$ iff $x^*=0$. 
Furthermore, when the drift $m$ exists, then $m\in K$ iff $x^*=0$. 
Hence, if we want that $\esp{\tau_x}=\infty$, we should restrict to the case $L(x^*)=1$.

%\begin{equation}\label{eqn}
%\begin{split}
%& \\
%& . 
%\end{split}\end{equation}

\bibliography{OsBib}
\bibliographystyle{amsalpha}
\end{document}